\documentclass[a4paper]{amsart}
\usepackage{amsmath,amsfonts,amssymb}
\usepackage{verbatim}
\usepackage{enumerate}
\usepackage{url}
\usepackage[colorlinks]{hyperref}
\usepackage[latin1]{inputenc}
\usepackage[english]{babel}
\usepackage{tikz}
\usepackage[margin=2.5cm]{geometry}

\def\Z{\mathbb Z}

\def \F{\mathbb F}

\def\ord{\mathop{\rm ord}\nolimits}
\def\rad{\mathop{\rm rad}}

\theoremstyle{plain}
\newtheorem{theorem}{Theorem}[section]
\newtheorem{lemma}[theorem]{Lemma}
\newtheorem{definition}[theorem]{Definition}
\newtheorem{corollary}[theorem]{Corollary}

\def\qed{\hfill\hbox{$\square$}}

\theoremstyle{definition}

\author[N. E. Ar\'evalo Baquero]{Nelcy Esperanza Ar\'evalo Baquero}
\address{
Departamento de Matem\'{a}tica\\
Universidade Federal de Rio Grande do Sul\\
UFRGS\\
Porto Alegre, RS \\
 91509-900\\
 Brazil\\
 }
 \email{nearevalob@unal.edu.co}
\author[F. E. Brochero Mart\'{\i}nez]{F. E. Brochero Mart\'{\i}nez}
\address{
Departamento de Matem\'{a}tica\\
Universidade Federal de Minas Gerais\\
UFMG\\
Belo Horizonte, MG\\
 31270-901\\
 Brazil\\
 }

 \email{fbrocher@mat.ufmg.br }

\title{ Factorization of  Dickson polynomials over Finite Fields}

\date{\today
}

\subjclass[2000]{ }


\begin{document}
\maketitle
\begin{abstract}
Let $D_n(x;a)$ and $E_n(x;a)\in \F_q[x]$  be  Dickson polynomials  of first and second kind respectively, where $\F_q$ is a finite field with $q$ elements.  In this article we show explicitly the irreducible factors these polynomials in the case that every prime divisor of $n$ divides $q-1$.
This result generalizes the results find in \cite{Chou},\cite{FiYu}, \cite{Tosun} and  \cite{Tosun1}.

\end{abstract}
\section{Introduction}
The Dickson polynomials $D_n(x;a)$ over finite fields was introduce in the nineteenth century  by Leonard E. Dickson as part of his PhD thesis.
These  polynomials have several interesting applications and properties, being mainly examples of families of permutation polynomials and  now $ D_{n} (x, a) \in \F_q [x] $ are  called Dickson polynomials of the first kind, to distinguish them from their variations introduced by Schur in 1923, which are actually  called  the Dickson polynomials of the second kind $E_{n} (x, a) \in\F_q [x] $.

When $ a = 0$,  $ D_n (x, 0) = x^n$ 
is a permutation polynomial  (PP)  in $ \F_q $, if and only if, $\gcd(n, q-1) = 1 $. In the case that  $ a\in \F_q^* $, it is known that the Dickson polynomial $ D_n (x, a) $ induces a permutation of $\F_q $, if and only if $\gcd(n, q^2 -1) = 1$ (see {\cite[Theorem 7.16]{LiNi}} or {\cite[Theorem 3.2] {Lidl2}}). This simple condition provides a very effective test for determining which polynomials $ D_n (x, a) $ induce permutations of $\F_q $, and furthermore, once the condition is satisfied, we obtain  $q-1 $ different permutations, one for each of the elements $ a \in \F_q^*$.
These polynomials have several applications and interesting properties.  In fact,  with the digital advances,  practical  applications  of permutation polynomials can be found in cryptography, combinatorial designs, error-correcting codes, as well as hardware implementation of turbo decoders, feedback shift-register, linear-feedback shift register, among other applications , as well as pure theoretical results.

In the field of complex numbers, Dickson polynomials are essentially equivalent to the classical Chebyshev polynomials $ T_n (x) $  with a simple change of variable. In fact, Dickson polynomials are sometimes referred as  Chebyshev Polynomials  in mathematical  literature.
They were rediscovered by Brewer , who used certain Dickson polynomials of the first kind to calculate  Brewer sums (see \cite{Alaca} and \cite{brewer}).
 The Dickson polynomials  are also related in some way to Kloosterman sum;  in \cite{Moi}  Moisio proved that,  if  Kloosterman sum  $K_{q^n}(a)$ is different to $1$ for some $a\in F_q^*$, then the minimal polynomial $f(x)$ of $K_q(a)$ over $\Z$ must be a factor of $D_n(x,q)+(-1)^{n-1}$. 

In recent years, these polynomials have been  object  of  intense studies; in  \cite{Lidl2}, the authors  give  a comprehensive survey  about Dickson polynomials, including applications such as  Dickson cryptosystems, Dickson pseudoprimes analogous to Carmichael numbers, etc. In particular, the problem of factorization of these polynomials have been studied by several authors. 
For example, in \cite{GaMu}, Gao and Mullen study the irreducibility  of $D_n(x,a)+b$, where $n$ is odd (see also \cite{Tur}). Their argument is based on a well-known  irreducibility criterion for binomial $x^t-\alpha$ over finite fields. 
Chou \cite {Chou} and later  Bhargava and Zieve \cite{Bhargava} study a factorization of the Dickson polynomials of the first kind on $ \F _q $ using methods more complicated   than those we use here, because even though the factorization found by  Bhargava and Zieve   is made in $\F_q $, the factors contain elements outside $ \F _q $. 
Fitzgerald and Yucas \cite {FiYu}  give factorizations of cyclotomic polynomials $Q_{3\cdot2^n}(x)$  for all $\F_q$ of characteristic not igual to 2 or 3. They apply this result to  get explicit factorizations of Dickson polynomials $D_{3\cdot 2^n}(x)$ and $E_{3\cdot 2^n-1}(x)$, respectively.
Tosun \cite{Tosun} extends the previous one by obtaining explicit factors of Dickson polynomials of the first  and second kind  $D_{3 \cdot   2^m}(x, a)$ and  $E_{3\cdot 2^n-1}(x,a)$ over $ \F _q $, where $a$ be an arbitrary element of $\F_q$ and odd characteristic.

In this work, we study the problem of splitting $ D_n (x, a) $ and $ E_n (x, a) $ into  irreducible factors   over $ \F _q $, where  $ \F _q $ is a finite field with  $ q $ elements, $n$  is a positive integer some that every prime divisor of $n$ divides $q-1$ for Dickson polynomials of the first kind and every prime divisor of $n+1$ divide $q-1$ for Dickson polynomials of the second kind.  
The result will be divide in several  cases  depending  essentially on  the class  of $q$ modulo $4$ and also if $a$ is an square in $\F_q$ . 

We note that when $n$ is a power of $2$, the condition $\rad(n)|(q-1)$ is trivial and then the results in \cite{Tosun} are  particular case of our results.

The structure of this paper is as follows. In Section 2, we give a formal definition of Dickson polynomials with parameter $a$, some useful properties of that polynomial, and  in particular,  the $a$-self reciprocal property. In addition, we show, without proof, some results about the factorization of polynomial of the form $x^{m}\pm 1$.   The factorization of Dickson polynomials  of the first kind in odd characteristic is determined in Section 3.  In Section 4, we give the factorization of Dickson polynomials of the second  kind in odd characteristic. In Section 5, we provide a factorizations of Dickson polynomials in even characteristic.

\section{ Preliminaries}
Throughout this paper, $\F_q$ will denote the finite field with $q$ elements, where $q$ is a power of a prime $p$.  For any $a\in \F_q^*$, $\ord_q(a)$ will denote the order of $a$ in the cyclic group $\F_q^*$ and for each positive integer $n\ge 2$, $\rad(n)$ denotes the product of every prime factor that divides $n$. 
The Dickson polynomials of first and second kind are defined as:

\begin{definition}
Let  $n\geq 1$ be an integer and $a\in\mathbb{F}_{q}$. The polynomial  $D_{n}(x,a)\in\mathbb{F}_{q}[x]$ defined as \begin{equation}
D_{n}(x,a)=\sum_{i=0}^{\lfloor n/2\rfloor}\dfrac{n}{n-i} {n-i \choose i }(-a)^{i}x^{n-2i}
  \label{eqn:Dn2}
\end{equation}
is called  the $n$-th Dickson polynomial of the first kind with parameter $a$.
\end{definition}

\begin{definition}
Let  $n\geq 1$ be an integer and $a\in\mathbb{F}_{q}$. The polynomial  $E_{n}(x,a)\in\mathbb{F}_{q}[x]$ defined as \begin{equation}
E_{n}(x,a)=\sum_{i=0}^{\lfloor n/2\rfloor} {n-i \choose i }(-a)^{i}x^{n-2i}
  \label{eqn:Dn2}
\end{equation}
is called  the $n$-th Dickson polynomial of the second kind with parameter $a$.
\end{definition}

The $n$-th Dickson polynomials $D_n(x;a)$ and $E_n(x;a)$ are  the unique degree $n$ polynomials satisfying  the formal  relations 
$$D_n\left(y+\frac 1y;a\right)=y^n+\frac{a^n}{y^{n}}\quad\text{and}\quad  E_n\left(y+\frac 1y;a\right)=\dfrac{y^n-\dfrac {a^n}{y^n}}{y-\dfrac ay},$$
respectively. These equalities are known as Waring's Identities.

In the case when $a=1$, we denote the polynomial $D_n(x;1)$ and $E_n(x;1)$ as $D_n(x)$ and $E_n(x)$, respectively. In this case, these polynomials can be see as polynomial with integer coefficients and then   $D_{pn}(x)=(D_n(x))^p$ and $E_{pn}(x)=(E_n(x))^p$, where $p$ is the characteristic of the field. 

The following lemma show some interesting properties of Dickson polynomials, that will be useful in the next sections.

\begin{lemma}[\cite{Lidl2}, Lemma 2.6]\label{proptipo1}
The Dickson polynomials $D_{n}(x,a)$ and $E_n(x,a)$ satisfy the following properties
\begin{enumerate}[(i)]
\item $D_{mn}(x,a)=D_{m}(D_{n}(x,a),a^{n})$ for $m\geq 0$ and $n\geq 0$,
\item $D_{np^{r}}(x,a)=[D_{n}(x,a)]^{p^{r}}$ for $n\geq 0$, $r\geq 0$, where $p$ is the characteristic of $\mathbb{F}_{q}$,
\item $b^{n}D_{n}(x,a)=D_{n}(bx,b^{2}a)$ for $n\geq 0$,
\item $b^{n}D_{n}(b^{-1}x,a)=D_{n}(x,b^{2}a)$ for $n\geq 0$ and $b\neq 0$.
\item $E_{n}(x,a)=[E_{m}(x,a)]^{p^{r}}(x^{2}-4a)^{\frac{p^{r}-1}{2}}$ for  $n\geq 0$ and $r\geq 0$, where  $p$ is the characteristic of  $\mathbb{F}_{q}$ and $n+1=(m+1)p^{r}$,
\item $b^{n}E_{n}(x,a)=E_{n}(bx,b^{2}a)$ for $n\geq 0$,
\item $b^{n}E_{n}(b^{-1}x,a)=E_{n}(x,b^{2}a)$ for $n\geq 0$ and $b\neq 0$.
\end{enumerate}
\end{lemma}

The following definition is borrowed from \cite{FiYu3} and \cite{FiYu2},  that is also one possible generalization of the notion of reciprocal polynomial. 

\begin{definition} Let $a$ be an element in $\F_q^*$.  For each monic  polynomial $f(x)\in \F_q[x]$ of degree $n$ with $f(0)\ne 0$,  let define $f_a^*(x)$, the $a$-reciprocal of $f (x)$, by 
$$f_a^*(x) =\frac {x^n}{f(0)} f\left(\frac ax\right).$$
In the case when $f(x)=f_a^*(x)$ we say that $f(x)$ is $a$-self reciprocal. 
\end{definition}

It is easy to prove that $f_a^*(x)$ is a monic polynomial  and if $\alpha\in \overline \F_q$ is a root of $f(x)$, then $\frac a{\alpha}$ is a root of $f_a^*(x)$. In addition  
$$(f\cdot g)_a^*(x)=f_a^*(x)\cdot g_a^*(x)\quad\text{and}\quad (f_a^*)_a^*(x)=f(x),$$ 
for any $f$ and $g$ monic polynomials with $(f\cdot g)(0)\ne 0$.  In particular, the polynomial $f(x)$ is an irreducible polynomial in $\F_q[x]$ if and only if $f_a^*(x)$ is  also irreducible.

\begin{lemma}
Let $f(x)$ be  a $a$-self reciprocal polynomial of degree $n=2m$ in  $\mathbb{F}_{q}$. Then $f(x)$ can be write as
\[f(x)=b_{m}x^{m}+\sum_{i=0}^{m-1}b_{2m-i}(x^{2m-i}+a^{m-i}x^{i}),\]
where  $b_{j}\in\mathbb{F}_{q}$ for every $j=0,1,\ldots, m$.
\end{lemma}
\begin{proof}
The result follows directly from comparing the coefficient of $f$ and $f_a^*$.\qed
\end{proof}

\begin{definition}\label{psiphi}
Let $P_{m}$ be the family of monic polynomials of degree $n$  in $\mathbb{F}_{q}$ and $S_{2m,a}$  be the family of  monic polynomials  $a$-self reciprocal of degree $2m$ in  $\mathbb{F}_{q}$. 
For each positive integer $m$, let
\begin{equation*}
\Phi_{a}:P_{m}\rightarrow S_{2m,a} \quad\text{and}\quad \Psi_{a}:S_{2m,a}\rightarrow P_{m}
\end{equation*}
be the applications defined as
$$\Phi_{a}(f(x))=x^{m}f\left(x+\dfrac{a}{x}\right)$$
and 
\begin{eqnarray*}
\Psi_{a}(g(x))&=&\Psi_{a}\left(b_{m}x^{m}+\sum_{i=0}^{m-1}b_{2m-i}(x^{2m-i}+a^{m-i}x^{i})\right)\\
&=& b_{m}+\sum_{i=0}^{m-1}b_{2m-i}D_{m-i}(x,a).
\end{eqnarray*}
\end{definition}

\begin{theorem}[\cite{FiYu2} Theorem 3.1.]\label{teo3} Let $a$ be an element in $\F_q^*$ and $\Phi_a$ and $\Psi_a$ be as in  Definition \ref{psiphi}. Then 
\begin{enumerate}[a)]
\item  $\Phi_a\circ \Psi_a=Id_{S_{2m,a}}$ and $\Psi_a\circ \Phi_a=Id_{P_{m}}$
\item $\Phi_a$ and $\Psi_a$ are multiplicative functions.
\item If $f(x)$ is a monic irreducible non-trivial $a$-self reciprocal polynomial of degree $2m$, then $\Phi_a(f(x))$ is an irreducible polynomial. 
If  $g(x)$ is an irreducible polynomial of degree $m$ and non a-self
reciprocal, then $\Psi_a (g (x) g_a^* (x))$ is irreducible.
\end{enumerate}
\end{theorem}

%

\begin{corollary}\label{lema4}
Let $\Phi_{a}$ be as in the definition \ref{psiphi}. Then
\[\Phi_{a}(D_{n}(x,a))=x^{2n}+a^{n}.\]
\end{corollary}
\begin{proof}
\[\Phi_{a}(D_{n}(x,a))=x^{n}D_{n}\left(x+\dfrac{a}{x},a\right)=x^{n}\left(x^{n}+\left(\dfrac{a}{x}\right)^{n}\right)=x^{2n}+a^{n}.\]\qed
\end{proof}

From this corollary, in order to find the irreducible factors of $D_{n}(x,a)$, it is enough to split into irreducible factors the polynomial $x^{2n}+a^{n}$. The factorization of this last polynomial has been  extensively studied (see \cite{BGM}, \cite{BGO}, \cite{BRS}, \cite{CLT}, \cite{FiYu} ),  then we show, without proof,  some results that will be useful in the following section.

\begin{theorem}[\cite{LiNi} Theorem 3.35] \label{irreducible} Let $n$ be a positive integer and $f(x)\in   \F_q[x]$ be an irreducible polynomial of degree $m$ and exponent $e$.  Then the polynomial $f(x^n)$ is irreducible over $  \F_q$ if and only if the following conditions are satisfied:
\begin{enumerate}
\item $rad(n)$ divides $e$;
\item gcd$(n, (q^m-1)/e)=1$ and
\item  if $4|n$ then $4|q^m-1$.
\end{enumerate}
In addition, in the case when the polynomial $f(x^n)$ is irreducible, then it has degree $mn$ and exponent $en$.
\end{theorem}

\begin{theorem}[\cite{BGO} Theorem 1]\label{muitoimp}
Let $\mathbb{F}_{q}$ be a finite field and  $n\in\mathbb{N}$ such that 
\begin{enumerate}[a)]
\item $q\equiv 1\pmod 4$ or $8\nmid n$,
\item $\rad(n)$ divides $q-1$.
\end{enumerate}
Then every irreducible factor of $x^{n}-1$ is of the form  $x^{t}-a$, where  $t$ divides $\frac{n}{\gcd(n,q-1)}$, $a\in\mathbb{F}_{q}$ and $ord_{q}(a)$ divides $\gcd(\frac{n}{t},q-1)$.
\end{theorem}

\begin{theorem}[\cite{BGO} Theorem 2]\label{copolirr}
Let  $q$ and $n$ such that $q\equiv 3\pmod 4$, $8\mid n$ and $rad(n)\mid (q-1)$, then  every irreducible factor of  $x^{n}-1$ in  $\mathbb{F}_{q}[x]$ is one of the following types
\begin{enumerate}[i)]
\item $x^{t}-a$, where  $a\in\mathbb{F}_{q}$, $t$ divides $\frac {n}{\gcd(n,q^2-1)}$  and $ord_{q}a$ divides $\gcd(\frac{n}{t},q-1)$,
\item $x^{2t}-(b+b^{q})x^{t}+b^{q+1}$, where  $b\in\mathbb{F}_{q^{2}}\setminus\mathbb{F}_{q}$,  $t$ divides $\frac {n}{\gcd(n,q^2-1)}$ and $ord_{q^{2}}b$ divides $\gcd(\frac{n}{t},q^{2}-1)$.
\end{enumerate}
\end{theorem}

\section{Factorization of Dickson polynomials of the first kind }
Throughout this section,   $n$ represents a positive integer such that $\rad(n)$ divides $q-1$, i.e.,  every prime divisor of $n$ also divides $q-1$.  In the case that $a$ is an square in $\F_q$,  we can see that, by a linear change of  the variable,  the factorization of $D_n(x,a)$ and $D_n(x)$ are equivalent. In addition, the factorization of $D_n(x)$ also depend on the class of $q$  modulo $4$.  Thus, our result will divide in three cases.  

\begin{theorem}\label{teofator1}
Let  $a\in\mathbb{F}_q^*$ be a square  in $\mathbb{F}_q$ and  assume that  either  $q\equiv 1\pmod 4$ or $n$ is an odd positive integer. If $\rad(n)|(q-1)$,  then every irreducible factor of $D_{n}(x;a)$ in  $\mathbb{F}_q$  is of the form  $D_{t}(x,a)-b^{t}(\alpha+\alpha^{-1})$, where  $b^2=a$, $\alpha\in\mathbb{F}_{q}^*$ and $t$  is a divisor of $\frac{4n}{\gcd(4n,q-1)}$,  such that the following conditions are satisfied
\begin{enumerate}[(i)]
\item $\alpha^{\frac{2n}{t}}=-1$,
\item $\rad(t)\mid ord_{q}(\alpha)$,
\item $\gcd \left(t,\frac{q-1}{ord_{q}(\alpha)}\right)=1$.
\end{enumerate} 
\end{theorem}

\begin{proof}
From  Lemma   \ref{proptipo1} item {(iv)}, we know that   $D_{n}(x,a)=b^{n}D_{n}(b^{-1}x,1)$ for $n\geq 0$ and $b\neq 0$. 
Then, the factorization of $D_{n}(x,a)$ can be obtained from the factorization of $D_{n}(y)$ where $y=b^{-1}x$.
In addition, by Corollary \ref{lema4}, we know that $\Phi_1(D_n(y))= y^{2n}+1$ and since $y^{4n}-1=(y^{2n}-1)(y^{2n}+1)$, if we find the irreducible factor of $y^{4n}-1$, in particular we obtain the irreducible factors of $y^{2n}+1$. 

It follows from  Theorem  \ref{muitoimp}  that every irreducible factor of $y^{4n}-1$  is of the form $y^{t}-\alpha$, where $t$ divides $ \frac{4n}{gcd(4n,q-1)}$ and $\alpha$ is an appropriate element of $\F_q$. 
In fact,  $y^{t}-\alpha$ divides $y^{2n}+1$  is equivalent to  $y^{2n}\equiv -1\pmod {y^{t}-\alpha}$, and therefore $y^{2n}\equiv \left(y^{t}\right)^{\frac{2n}{t}}\equiv -1\pmod {y^{t}-\alpha}$.  Consequently $\alpha^{\frac{2n}{t}}=-1$.

Let suppose that $f(y)$ is an irreducible factor of $D_{n}(y)$ and  $h(y)$ be the image of $f(y)$ by the map $\Phi_{1}$.  Hence, $h(x)$ is a factor of $x^{2n}+1$, not necessarily irreducible in $\F_q[y]$.  There  exists an irreducible factor of the form $x^t-\alpha$  that divides $h(x)$. Since $h(x)$ is self-reciprocal polynomial, it follows that the reciprocal $x^t-\alpha^{-1}$ also divides $h(x)$.  

At this point, we have two cases to consider

\begin{enumerate}[(a)]
\item If $\alpha\neq\alpha^{-1}$, then $\left(y^{t}-\alpha\right)\left(y^{t}-\alpha^{-1}\right)\mid h(y)$. Since
$$\left(y^{t}-\alpha\right)\left(y^{t}-\alpha^{-1}\right)=
y^{t}\Bigl(y^{t}+\frac{1}{y^{t}}-(\alpha+\alpha^{-1})\Bigr)=
y^{t}\Bigl(D_{t}\Bigl(y+\frac{1}{y}\Bigr)-(\alpha+\alpha^{-1})\Bigr),$$
we get that  $\Psi_{1}((y^{t}-\alpha)(y^{t}-\alpha^{-1}))= D_{t}(y)-(\alpha+\alpha^{-1})$ divides 
$f(y)$.
Now, using the fact that $f(y)$ is a monic  irreducible  polynomial, we conclude that $f(y)=D_{t}\left(y\right)-\left(\alpha+\alpha^{-1}\right)$.  Finally, rewrite this identity using the original variable we obtain
$f(b^{-1}x)=D_{t}\left(b^{-1}x,1\right)-\left(\alpha+\alpha^{-1}\right)=b^{-t}D_{t}\left(x,a\right)-\left(\alpha+\alpha^{-1}\right)$ and therefore
$$b^{t}f(b^{-1}x)=D_{t}\left(x,a\right)-b^{t}\left(\alpha+\alpha^{-1}\right)$$
is a monic irreducible factor of $D_{n}\left(x,a\right)$ in $\mathbb{F}_q$.

\item In the case when  $\alpha=\alpha^{-1}$, we have that  $\alpha=\pm 1$ and  $y^{t}-\alpha=y^{t}\pm 1$. It is clear that this polynomial is irreducible if $t=1$. In the case that $t\ne 1$, we have that   $y^t- 1$ is always reducible and  $y^t+1$  is reducible if $t$ has an odd prime divisor.     Lastly, if $t$ is a power of $2$, since $q\equiv 1\pmod 4$, we have that $-1$ is an square in $\F_q$ and  hence $y^t+1$ is also reducible.  
In any case we conclude that  $y\pm 1$ divides $y^{2n}+1$, thus $y^{2n}+1\equiv (\mp 1)^{2n}+1 \equiv 2 \equiv 0\pmod {y\pm 1}$ implies that $char(\mathbb{F}_q)=2$, which is a contradiction. \qed
\end{enumerate}
\end{proof}

In the previous  theorem, we consider the case when $a\in\mathbb{F}_q$ is an square in $\mathbb{F}_q$ and  either  $q\equiv 1\pmod 4$ or $n$ is an odd integer.  In the following theorems, we use this result  in order to  understand the complementary cases  of this conditions.

\begin{theorem}\label{teoimp31}
Let  $a$ be an square in $\mathbb{F}_q$, $q\equiv 3\pmod 4$ and  $n$ an even positive integer such that $\rad(n)|(q-1)$.
Let $b\in  \F_q$ such that $b^2=a$, $\alpha\in\mathbb{F}_{q^2}^*$ and $t$ be a divisor of  $\frac{4n}{\gcd(4n, q^2-1)}$, such that  $t$ and $\alpha$ satisfy  the conditions \textit{(i)}, \textit{(ii)} and  \textit{(iii)} of the Theorem \ref{teofator1}  in the finite field $\F_{q^2}$.
Then every irreducible factor of $D_{n}(x,a)$ in $\mathbb{F}_q$ is of the following types
\begin{enumerate}[(a)]
\item $D_{t}(x,a)-b^{t}(\alpha+\alpha^{-1})$  in the case that $\alpha\in\mathbb{F}_{q}^*$ or $\alpha^{q+1}=1$,
\item $\left(D_{t}(x,a)-b^{t}(\alpha+\alpha^{-1})\right)\left(D_{t}(x,a)-b^{t}(\alpha^{q}+\alpha^{-q})\right)$   otherwise.
\end{enumerate} 
\end{theorem}

\begin{proof}
We observe that every  irreducible factor of $D_n(x,a)$ in $\F_q[x]$ is also a factor, not necessarily irreducible, in $\F_{q^2}[x]$. Thus, at first time we  consider the Dickson polynomial $D_{n}(x,a)$  as a polynomial  in $\mathbb{F}_{q^2}[x]$.  
Since  $q^2\equiv 1\pmod 4$, by Theorem \ref{teofator1}, every irreducible factor of $D_n(x,a)$ in $\mathbb{F}_{q^2}[x]$ is of the form  $D_{t}\left(x,a\right)-b^{t}\left(\alpha+\alpha^{-1}\right)\in \mathbb{F}_{q^2}[x]$.

Some of these factors are also in $\F_q[x]$, but that happens when  $\alpha+\alpha^{-1}\in \mathbb{F}_{q}$. This condition is equivalent to $\alpha+\alpha^{-1}=\left(\alpha+\alpha^{-1}\right)^{q}$ and that equation is equivalent to  $(\alpha^{q+1}-1)(\alpha^{q-1}-1)=0$. Thus $\alpha\in\mathbb{F}_{q}^*$ or $\alpha^{q+1}=1$. 

In the case when none of these conditions are satisfied, $D_{t}\left(x,a\right)-b^{t}\left(\alpha+\alpha^{-1}\right)\in\mathbb{F}_{q^2}[x]$ is a factor of $D_{n}(x,a)$, that is not in $\F_q[x]$. From the fact that the coefficient of $D_n(x,a)$ is invariant by the  Frobenius Automorphism
\begin{eqnarray*}
\tau:\mathbb{F}_{q^2} & \rightarrow & \mathbb{F}_{q^2}\\
\beta & \mapsto & \beta^{q},
\end{eqnarray*}
we conclude that  $D_{t}\left(x,a\right)-b^{t}\left(\alpha^{q}+\alpha^{-q}\right)$ is also an irreducible factor of  $D_{n}(x,a)$. In addition, from the fact that  $\alpha+\alpha^{-1}\neq\alpha^{q}+\alpha^{-q}$, we concluded that $\left(D_{t}(x,a)-b^{t}(\alpha+\alpha^{-1})\right)\left(D_{t}(x,a)-b^{t}(\alpha^{q}+\alpha^{-q})\right)$ is a factor of  $D_{n}(x,a)$.

Lastly,  the coefficients of  $\left(D_{t}(x,a)-b^{t}(\alpha+\alpha^{-1})\right)\left(D_{t}(x,a)-b^{t}(\alpha^{q}+\alpha^{-q})\right)$ are invariant by $\tau$, so $$\left(D_{t}(x,a)-b^{t}(\alpha+\alpha^{-1})\right)\left(D_{t}(x,a)-b^{t}(\alpha^{q}+\alpha^{-q})\right)\in \F_q[x]$$
 is an irreducible  factor of  $D_{n}(x,a)$   in the polynomial ring $\F_q[x]$.\qed
\end{proof}


\begin{theorem}\label{teoimpult}
Let  $a\in\mathbb{F}_q$ be a non-square in $\mathbb{F}_q$, $b\in\mathbb{F}_{q^2}\setminus\mathbb{F}_{q}$ such that $b^2=a$ and $n$ be a positive integer such that $\rad(n)|(q-1)$. Let 
 $\alpha\in\mathbb{F}_{q^2}^*$ and  $t$ be a divisor of $\frac{4n}{\gcd(4n, q^2-1)}$ satisfied the conditions \textit{(i)}, \textit{(ii)} and \textit{(iii)} of Theorem \ref{teofator1} in the field $\mathbb{F}_{q^{2}}$. 
Then every irreducible factor of $D_{n}(x,a)$ in $\mathbb{F}_q$ is one of the following forms
\begin{enumerate}[(a)]
\item $D_{t}(x,a)-b^{t}(\alpha+\alpha^{-1})$ in the case that  $t$ is even and either $\alpha\in\mathbb{F}_{q}^*$  or $\alpha^{q+1}=1$,  or $t$ is odd  and either $\alpha^{q-1}=-1$ or $\alpha^{q+1}=-1$,
\item $\left(D_{t}(x,a)-b^{t}(\alpha+\alpha^{-1})\right)\left(D_{t}(x,a)-b^{qt}(\alpha^{q}+\alpha^{-q})\right)$  otherwise
\end{enumerate}
\end{theorem}

\begin{proof}
As in the previous  theorem, we  consider the Dickson polynomial $D_{n}(x,a)$  as a polynomial  in $\mathbb{F}_{q^2}[x]$. 
In any case, we have that $q^2 \equiv 1\pmod 4$ and $a$ is an square in $\mathbb{F}_{q^2}$, i.e.,   $a=b^2$ where $b\in\mathbb{F}_{q^2}\setminus\mathbb{F}_{q}$.  
It follows from  Theorem \ref{teofator1}  that every irreducible factor of $D_n(x,a)$ in $\mathbb{F}_{q^2}[x]$ is of the form  $D_{t}(x,a)-b^{t}(\alpha+\alpha^{-1})\in\mathbb{F}_{q^2}[x]$. 

At this point, we need to determine which factors are in $\mathbb{F}_{q}[x]$. Since $D_{t}(x,a)\in\mathbb{F}_{q}[x]$, we need find conditions in order that $b^{t}\left(\alpha+\alpha^{-1}\right)\in\mathbb{F}_{q}$, which is equivalent to   $b^{t}\left(\alpha+\alpha^{-1}\right)=b^{qt}\left(\alpha+\alpha^{-1}\right)^{q}$. 
This equation can be rewritten as  $b^{t(q-1)}\left(\alpha^{q}+\alpha^{-q}\right)=\alpha+\alpha^{-1}$.

It follows from  $a$  not being an square in $\F_q$ that $b^{q-1}=\left(b^{2}\right)^{\frac{q-1}{2}}=a^{\frac{q-1}{2}}=-1$.
At this point, we have two cases to consider

\begin{enumerate}
\item  When $t$ is even:   $b^{t(q-1)}\left(\alpha^{q}+\alpha^{-q}\right)=\alpha+\alpha^{-1}$ if and only if $\alpha^{q}+\alpha^{-q}=\alpha+\alpha^{-1}$, and so $\alpha\in\mathbb{F}_{q}^*$ or $\alpha^{q+1}=1$.
\item When $t$ is odd: $b^{t(q-1)}\left(\alpha^{q}+\alpha^{-q}\right)=\alpha+\alpha^{-1}$ if and only if $-\alpha^{q}-\alpha^{-q}=\alpha+\alpha^{-1}$, that is equivalent to $(\alpha^{q-1}+1)(\alpha^{q+1}+1)=0$.
\end{enumerate}

Finally,  in the case when none of the conditions are satisfied, the factor $D_{t}\left(x,a\right)-b^{t}\left(\alpha+\alpha^{-1}\right)\in\mathbb{F}_{q^2}[x]$ of  $D_{n}(x,a)$ is not in $\F_q[x]$. 
Using the Frobenius Automorphism we also conclude that  $D_{t}\left(x,a\right)-b^{qt}\left(\alpha^{q}+\alpha^{-q}\right)$ is an irreducible factor of $D_{n}(x,a)$ in $\F_{q^2}[x]$ and 
$$\left(D_{t}(x,a)-b^{t}(\alpha+\alpha^{-1})\right)\left(D_{t}(x,a)-b^{qt}(\alpha^{q}+\alpha^{-q})\right)$$ 
is an irreducible factor of  $D_{n}(x,a)$  in $\mathbb{F}_q[x]$. \qed
\end{proof}

\section{Dickson polynomials of the second kind}
Throughout this section,   $n$ represents a positive integer such that $\rad(n+1)$ divides $q-1$, i.e.,  every prime divisor of $n+1$ also divides $q-1$.

\begin{lemma}\label{lema5}
Let  $\Phi_{a}$ be the function defined at \ref{psiphi}.  Then 
\[\Phi_{a}(E_{n}(x,a))=\dfrac{x^{2(n+1)}-a^{n+1}}{x^{2}-a}.\]
\end{lemma}

\begin{proof} Using the Waring's Identity for $E_n$, we have
\[\Phi_{a}(E_{n}(x,a))=x^{n}E_{n}\left(x+\dfrac{a}{x},a\right)=x^{n}\left(\dfrac{x^{n+1}-\dfrac{a^{n+1}}{x^{n+1}}}{x-\dfrac{a}{x}}\right)=\dfrac{x^{2(n+1)}-a^{n+1}}{x^{2}-a},\]
as we want to prove.\qed
\end{proof}

Using the previous  lemma and essentially the same steps of the proof of Theorems \ref{teofator1}, \ref{teoimp31} and \ref{teoimpult}, we obtain the following result,  enunciated  without proof.

\begin{theorem}\label{teofator2}
Let $a$  be an square in $\mathbb{F}_q^*$ with  $q\equiv 1\pmod 4$ or $n$ an even positive integer. In addition, $n$ satisfies that $\rad(n+1)$ divides $q-1$. Then every irreducible factor of  $E_{n}(x,a)$  in  $\mathbb{F}_q[x]$ is of the form $D_{t}(x,a)-b^{t}(\alpha+\alpha^{-1})$, where $b^2=a$, $\alpha\in\mathbb{F}_{q}^*$ and $t$  is a divisor of  $\frac{2(n+1)}{\gcd(2(n+1), q-1)}$,  satisfying the following conditions
\begin{enumerate}[(i)]
\item $\alpha^{\frac{2(n+1)}{t}}=1$,
\item $\rad(t)\mid ord_{q}(\alpha)$,
\item $\gcd\left(t,\frac{q-1}{ord_{q}(\alpha)}\right)=1$,
\item $(\alpha,t)\notin \{(1,1),(1,-1)\}$.
\end{enumerate} 
\end{theorem}


\begin{theorem}\label{teoimp32}
Let $a\in\mathbb{F}_q$ be a square in $\mathbb{F}_q^*$ with $q\equiv 3\pmod 4$ and $n$ an odd positive integer. In addition, $n$ satisfies that $\rad(n+1)$ divides $q-1$.
Let $b^2=a$, $\alpha\in\mathbb{F}_{q^2}^*$ and $t$ a divisor of  $\frac{2(n+1)}{\gcd(2(n+1), q-1)}$, satisfying  the conditions  {(i)}, {(ii)}, {(iii)} and {(iv)} of the previous  theorem in the field  $\mathbb{F}_{q^2}$. Then every irreducible factor of $E_{n}(x,a)$ in $\mathbb{F}_q$  has one of the following forms
\begin{enumerate}[(a)]
\item $D_{t}(x,a)-b^{t}(\alpha+\alpha^{-1})$ in the case that $\alpha\in\mathbb{F}_{q}^*$ or $\alpha^{q+1}=1$,
\item $\left(D_{t}(x,a)-b^{t}(\alpha+\alpha^{-1})\right)\left(D_{t}(x,a)-b^{t}(\alpha^{q}+\alpha^{-q})\right)$ in the case when  $\alpha$ does not satisfy  neither conditions of the item above. 
\end{enumerate} 
\end{theorem}


\begin{theorem}\label{teoimpult2}
Let  $a\in\mathbb{F}_q$ be a non-square in $\mathbb{F}_q$ and $n$ be an integer such that $\rad(n+1)$ divides $q-1$. Let   $b\in\mathbb{F}_{q^2}\setminus\mathbb{F}_{q}$ such that $b^2=a$, 
$\alpha\in\mathbb{F}_{q^2}^*$ and $t$ be a divisor  of $2(n+1)$,  satisfying  the conditions {(i)}, {(ii)}, {(iii)} and {(iv)}  of Theorem \ref{teofator2} in   $\mathbb{F}_{q^2}$.  Then every irreducible factor of  $E_{n}(x,a)$ in  $\mathbb{F}_q$ has  one of the following forms
\begin{enumerate}[(a)]
\item $D_{t}(x,a)-b^{t}(\alpha+\alpha^{-1})$ in the case  that $t$ is even and either $\alpha\in\mathbb{F}_{q}^*$ or  $\alpha^{q+1}=1$, or $t$ is odd  and either $\alpha^{q-1}=-1$ or $\alpha^{q+1}=-1$.
\item $\left(D_{t}(x,a)-b^{t}(\alpha+\alpha^{-1})\right)\left(D_{t}(x,a)-b^{qt}(\alpha^{q}+\alpha^{-q})\right)$ in the case that  $\alpha$ and $t$ do not satisfy  neither  conditions of the item above.
\end{enumerate}
\end{theorem}


\section{Dickson polynomial in characteristic 2}
Throughout this section,  $\F_q$ is a finite field of characteristic $2$. We note that if $n=2^{r}s$, where   $r\geq 0$ and $s$ odd,  then $D_{n}(x)=[D_{s}(x)]^{2^{r}}$, therefore, in order to splits  $D_{n}(x)$ into irreducible factors, it is enough to factorize $D_{s}(x,1)$. 
Hence, we  assume that  $n$ is positive integer such that $\rad(n)$ divides $q-1$. 

\begin{lemma}\label{lemacar2}
Let  $n=2m+1$ be a positive odd integer. 
\begin{enumerate}[(a)]
\item $D_{n}(x)=xF_{n}(x)^{2}$, for some polynomial $F_{n}(x)$ of degree  $m$, with  $F_{n}(0)\neq 0$,
\item $(x+1)\Phi_{1}(F_{n}(x))=x^{n}+1$.
\end{enumerate}
\end{lemma}

\begin{proof}
\begin{enumerate}[(a)]
\item Let define
\[F_{n}(x)=\sum_{i=0}^{m}\dfrac{n}{n-i} {n-i \choose i }x^{m-i}.\]
Then
\begin{eqnarray*}
xF_{n}(x)^{2}&=& x\left[\sum_{i=0}^{m}\dfrac{n}{n-i} {n-i \choose i }x^{m-i}\right]^{2}= x\sum_{i=0}^{m}\dfrac{n}{n-i} {n-i \choose i }x^{2m-2i}\\
&=& \sum_{i=0}^{m}\dfrac{n}{n-i} {n-i \choose i }x^{2m+1-2i}= \sum_{i=0}^{m}\dfrac{n}{n-i}{n-i \choose i }(-1)^{i} x^{n-2i}\\
&=& \sum_{i=0}^{\lfloor n/2\rfloor}\dfrac{n}{n-i} {n-i \choose i }(-1)^{i}x^{n-2i}= D_{n}(x,1).
\end{eqnarray*}
\item By Corollary \ref{lema4} item  {(i)} and  Theorem \ref{teo3} follows that
\[x^{2n}+1=\Phi_{1}(D_{n}(x,1))=\Phi_{1}(x)\Phi_{1}(F_{n}(x)^{2})=(x^{2}+1)\Phi_{1}(F_{n}(x))^{2}.\]
Therefore  $(x^{n}+1)^{2}=[(x+1)\Phi_{1}(F_{n}(x))]^{2}$.\qed
\end{enumerate}
\end{proof}

\begin{theorem}\label{teocar2tipo1}
Let  $\F_q$ be a finite field such that $char(\mathbb{F}_q)=2$, $a\in \F_q^*$ and  $n$ be a positive integer such that $\rad(n)|(q-1)$. Then every irreducible factor of $D_n(x,a)$ different than $x$, has multiplicity $2$ and is of the form  $D_{t}(x,a)-b^{t}(\alpha+\alpha^{-1})$, where  $b^2=a$, $\alpha\in\mathbb{F}_{q}^*$ and $t$ is a divisor of  $\frac{n}{\gcd(n,q-1)}$, satisfying the following conditions
\begin{enumerate}[(i)]
\item $\alpha^{\frac{n}{t}}=1$,
\item $rad(t)\mid ord_{q}(\alpha)$,
\item $gcd\left(t,\frac{q-1}{ord_{q}(\alpha)}\right)=1$,
\item $(\alpha,t)\neq (1,1)$.
\end{enumerate} 
\end{theorem}
\begin{proof}
In characteristic 2,  every element of $\F_q$ is an square. In addition, by Lemma \ref{proptipo1} and Lemma \ref{lemacar2}, we have  that
$$D_n(x,a)=b^{n} D_n(b^{-1}x,1)=b^{n-1}x (F_n(b^{-1}x))^2.$$
At this point, in order to find the factorization of $D_n(x,a)$, it is enough to find the factorization of $F_n(x)$. 
On the  other hand,  $\Phi_1(F_n(x))=\dfrac {x^n+1}{x+1}$, so from here the proof is  essentially the same proof of Theorem \ref{teofator1}.\qed 
\end{proof}

To conclude  this section, we analyse the factorization of Dickson polynomials of the second kind in characteristic $2$, that, as we shall see, is almost the same factorization of Dickson polynomials of the first kind. 
In  fact, if  $n+1=2^{r}(s+1)$ with  $r\geq 0$ and $s$ even, then  $E_{n}(x,1)=[E_{s}(x,1)]^{2^{r}}x^{2^{r}-1}$.
Therefore, in order to find the factors of  $E_{n}(x,1)$, we need to factorize   $E_{s}(x,1)$ with $s+1$ odd.  So, we can assume that  $n+1$ is odd. Then
$$E_{n}(x,1)=E_{n}\Bigl(y+\dfrac{1}{y},1\Bigr)=\dfrac{y^{n+1}-\Bigl(\dfrac{1}{y}\Bigr)^{n+1}}{y-\dfrac{1}{y}}=\frac{D_{n+1}\Bigl(y+\dfrac{1}{y},1\Bigr)}{y+\dfrac{1}{y}}=\frac{D_{n+1}(x,1)}{x}.$$
From Lemma \ref{lemacar2} item {(1),} we conclude that $E_{n}(x,1)=F_{n+1}(x)^{2}$, and therefore  the  factorization of $E_n(x,a)$ with the condition that $\rad(n+1)|(q-1)$ can be found using Theorem  \ref{teocar2tipo1}.

\end{document}